\documentclass[letterpaper, 10 pt, conference]{ieeeconf}  

\IEEEoverridecommandlockouts                              
\usepackage{amssymb}
\usepackage{tikz-cd}
\usetikzlibrary{arrows}
\usepackage{mathrsfs}

\usepackage{amsthm}
\usepackage{amsmath}
\usepackage{graphicx}
\usepackage{subcaption} 
\usepackage{enumerate}

\usepackage{centernot}

\newtheoremstyle{mystyle}
  {}
  {}
  {\itshape}
  {}
  {\bfseries}
  {.}
  { }
  {}

\theoremstyle{mystyle}

\newtheorem{theorem}{Theorem}[section]

\newtheorem{assumption}[theorem]{Assumption} 

\newtheoremstyle{mystyleUpshape}
  {}
  {}
  {}
  {}
  {\bfseries}
  {.}
  { }
  {}

\theoremstyle{mystyleUpshape}
\newtheorem{theoremUpshape}{Theorem}[section]

\newtheorem{remark}[theoremUpshape]{Remark}

\usepackage[backend=bibtex,
style=ieee,
bibencoding=ascii,
doi=false
]{biblatex}
\makeatletter
\newbox\dottedarrow@box
\setbox\dottedarrow@box\hbox
  {%
    \begin{tikzpicture}
      \draw[dotted,->] (0,0) -- (1.5em,0);
    \end{tikzpicture}%
  }
\newcommand*\dottedarrow
  {\relax\ifmmode\expandafter\dottedarrow@m\else\expandafter\dottedarrow@t\fi}
\newcommand*\dottedarrow@t[1][1.5em]
  {\resizebox{#1}{!}{\raisebox{.5ex}{\usebox\dottedarrow@box}}}
\newcommand*\dottedarrow@m[1][]
  {%
    \if\relax\detokenize{#1}\relax
      \mathchoice
        {\dottedarrow@t}
        {\dottedarrow@t}
        {\dottedarrow@t[1.1em]}
        {\dottedarrow@t[0.9em]}%
    \else
      \dottedarrow@t[#1]%
    \fi
  }
\makeatother

\title{\LARGE \bf
A generalized global Hartman-Grobman theorem for\\ asymptotically stable semiflows
}

\author{Wouter Jongeneel
\thanks{The author is with the KTH Royal Institute of Technology and Digital Futures (SE). The majority of this work was carried out when the author was a postdoctoral fellow at the department of Applied Mathematics (INMA, ICTEAM) at UCLouvain (BE), supported through the European Research Council (ERC) (grant agreement No. 864017 - L2C). Contact: \texttt{wouterjo@kth.se}, \texttt{www.wjongeneel.nl}.} 
}

\begin{document}

\maketitle
\thispagestyle{empty}
\pagestyle{empty}

\begin{abstract}
Recently, Kvalheim and Sontag provided a generalized global Hartman-Grobman theorem for equilibria under asymptotically stable continuous vector fields. By leveraging topological properties of Lyapunov functions, their theorem works without assuming hyperbolicity.
We extend their theorem to a class of possibly discontinuous vector fields, in particular, to vector fields generating asymptotically stable semiflows. 
\end{abstract}


\section{Introduction}
Linearization, in all its forms, remains one of the most powerful techniques to handle nonlinear sytems, \textit{e.g.}, there is machinery that is widely, rigorously and easily applicable. Unfortunately, one can typically only say something locally. Therefore, within the broader field of linearization techniques, Koopman operator theory is of great interest as, amongst other things, the analysis need not be local~\cite{ref:mezic2021koopman}.
However, in a recent survey on modern Koopman theory we find the following: ``\textit{..., obtaining finite-dimensional coordinate systems and embeddings in which the
dynamics appear approximately linear remains a central open challenge.}''~\cite[p. 1]{ref:brunton2022modernKoopman} and
``\textit{..., there is little hope for
global coordinate maps of this sort.}''~\cite[p. 4]{ref:brunton2022modernKoopman}.

Clearly, these comments relate to the desire of obtaining tools akin to the celebrated Hartman-Grobman theorem (\textit{e.g.}, see~\cite{ref:hartman1960lemma}); ideally, tools that are applicable globally and without relying on hyperbolicity. Indeed, a recent global extension of the Hartman-Grobman theorem by Kvalheim and Sontag exploits stability, and in particular Lyapunov theory, to overcome the restrictive reliance on hyperbolicity \textit{cf}.~\cite{ref:lan2013linearization,ref:eldering2018global,ref:kvalheim2021existence,ref:kvalheim2025global}. Here, the intuition is that the Lyapunov function replaces the Morse function. We remark that in the context of topological embeddings (\textit{i.e.}, not necessarily mapping onto a linear space), hyperbolicity has been relaxed before, also by exploiting stability, \textit{e.g.}, see \cite[Cor. 4]{ref:kvalheim2023linearizability}. 

Now, elaborating on \cite{ref:kvalheim2025global}, we can construct a similar result when the vector field is not necessarily continuous at the equilibrium point, see Theorem~\ref{thm:discont}. As in~\cite{ref:kvalheim2025global}, we build upon the topological results in \cite{ref:wilson1967structure} and \cite{ref:grune1999asymptotic}. Although motivated by the question of linearization, results like these sharpens in particular our understanding of topological equivalence. Differently put, these results further sharpen our understanding how to distinguish dynamical systems. To further motivate this particular setting, we look at dynamical systems that might fail to be continuous at the equilibrium for several reasons: this setting captures some closed-loop systems where discontinuous feedback was applied (which for several problems is necessary); this setting captures a class of systems that render the equilibrium finite-time stable; and this setting relates to studying robustness, \textit{e.g.}, through inclusions. We point the reader to, for instance, \cite{ref:cortes2008discontinuous} for more information. We already mention that we abstract this setting by simply saying: we study systems that generate asymptotically stable \textit{semi}flows, we return to this below.        

To illustrate our setting, the running example we have in mind is a normalized vector field on $\mathbb{R}^n$ of the form
\begin{equation}
\label{equ:normalized:flow}
    \dot{x} = \begin{cases} \displaystyle-\frac{x}{\|x\|_2} \quad &x\in \mathbb{R}^n\setminus\{0\}\\
    0 \quad &\text{otherwise},
    \end{cases}
\end{equation}
where the solutions will be understood \textit{in the sense of Filippov}\footnote{See~\cite{ref:hajek1979discontinuous} for a comparison of solution frameworks.}, with $\mathcal{F}$ denoting the Filippov operator \cite[p. 85]{ref:filippov1988differential}. For instance, the solution corresponding to~\eqref{equ:normalized:flow} becomes
\begin{equation}
\label{equ:flow:sol}
 (t,x)\mapsto\varphi_1^t(x) := \begin{cases}
 \left(1-\displaystyle\frac{t}{\|x\|_2}\right)x\quad &t\leq \|x\|_2 \\
     0 \quad &\text{otherwise},
 \end{cases}   
\end{equation}
which is not a flow, but a \textit{semi}flow. The reason being, $(t,x)\mapsto \varphi_1^t(x)$ need not be well-defined for $t<0$, \textit{e.g.}, consider $\varphi_1^{t}(0)$ for any $t<0$. Hence, we should restrict the domain of $\varphi_1$. More formally, given a topological space $M$, a continuous map $
\varphi:M\times \mathbb{R}_{\geq 0}\to M$, usually written as $(x,t)\mapsto \varphi^t(x)$, is said to be a \textit{semiflow} when $\varphi^0=\mathrm{id}_M$ and $\varphi^{s}\circ \varphi^t=\varphi^{s+t}$ for all $s,t\in \mathbb{R}_{\geq 0}$, that is, the identity- and the semi-group axiom are satisfied, as is the case for (semi)flows generated by smooth ODEs. For a thorough stability analysis of semiflows, we point the reader to \cite{ref:bhatiahajek2006local}.  

Indeed, as we will study semiflows that are possibly generated by vector fields, we consider sufficiently regular manifolds and not arbitrary topological spaces. Then, in general, given a semiflow $\varphi$ on some topological manifold $M$, we speak of a \textit{linearizing} (with respect to $\varphi$) \textit{homeomorphism} $h:M\to \mathbb{R}^n$ when $h\circ \varphi^t = e^{t\cdot A}\circ h$ for all $t\in \mathbb{R}_{\geq 0}$ and some matrix $A\in \mathbb{R}^{n\times n}$ (\textit{i.e.}. $\varphi^t$ and $e^{t\cdot A}$ are topologically conjugate). Observe that the Hartman-Grobman theorem provides us with a particularly convenient, but \textit{local}, linearization (\textit{e.g.}, $A$ is of the form $\mathrm{D}X(x_*)$ for some $C^1$ vector field $X$ and hyperbolic equilibrium point $x_*$). We emphasize that this type of linearization is a deliberate and convenient choice, but we could have focused on another canonical system or even another type of transformation. 

Before we continue, let us emphasize why these linearization questions are non-trivial. To that end, consider two scalar, asymptotically stable, linear ODEs $\dot{x}=-ax$ and $\dot{y}=-b y$, with $a,b>0$. Suppose we set $y=h(x):=\kappa x$, for some $\kappa\in \mathbb{R}\setminus \{0\}$. It follows that $\kappa(e^{-at}(\kappa^{-1}y))=e^{-bt}y$ must hold for all $t\geq 0$ and $y\in \mathbb{R}$, hence, $a=b$ must hold. Instead of a linear homeomorphism, one readily finds that, for instance 
\begin{equation*}
    x\mapsto h(x) :=
    \begin{cases}
        \mathrm{sgn}(x) |x|^{\frac{b}{a}} \quad &x\in \mathbb{R}\setminus \{0\} \\
        0 \quad &\text{otherwise}
    \end{cases}
\end{equation*}
does check out, \textit{i.e.}, $h(e^{-at}x)=e^{-bt}h(x)$. 

Now, when it comes to linearizing a system like~\eqref{equ:normalized:flow}, the first obstruction that comes to mind might be a lack of continuity at the equilibrium point, however, a more fundamental issue is that due to the finite-time stability property of such a semiflow, there cannot be a homeomorphism $h$ such that the conjugacy $h\circ \varphi_1^t\circ h^{-1}=e^{-t\cdot I_n}$ holds true for all $(t,h^{-1}(y))\in \mathrm{dom}(\varphi_1)$, for otherwise $\dot{y}=-y$ would correspond to a \textit{finite}-time stable system. With this observation in mind, in this short note we show to what extent semiflows corresponding to vector fields like~\eqref{equ:normalized:flow} can still be ``\textit{linearized}''. We emphasize that we will not consider a reparametrization of time.

\textit{Notation}. We let $B(0,r):=\{x\in \mathbb{R}^n:\|x\|_2<r\}$ be the open ball of radius $r$, centered at $0\in \mathbb{R}^n$, with $B(0,r)^c:=\mathbb{R}^n\setminus\{B(0,r)\}$ denoting its complement. With $\mathrm{cl}(W)$ and $\partial W$ we denote the topological closure and manifold boundary of $W$, respectively. The symbol $\simeq_t$ denotes topological equivalence, whereas $\simeq_h$ denotes homotopy equivalence. A continuous function $\gamma:\mathbb{R}_{\geq 0}\to \mathbb{R}_{\geq 0}$ is of class $\mathcal{K}_{\infty}$ when $\gamma(0)=0$, $\gamma$ is strictly increasing and $\lim_{s\to +\infty}\gamma(s)=+\infty$.

\section{Main result}
On $\mathbb{R}^n$, the vector fields we consider are possibly set-valued at $0$, locally essentially bounded on $\mathbb{R}^n$ and locally Lipschitz on $\mathbb{R}^n \setminus \{0\}$, \textit{e.g.}, like~\eqref{equ:normalized:flow}. It is known that under these assumptions, applying the Filippov operator, denoted $\mathcal{F}[\cdot]$, to such a vector field yields a map that is upper semi-continuous and compact, convex valued, allowing for a \textit{smooth} converse Lyapunov theory, \textit{e.g.}, see \cite{ref:clarke1998asymptotic}. Indeed, one can do with less regular vector fields. However, we focus on examples akin to~\eqref{equ:normalized:flow} and keep the presentation simple. 

For the appropriate generalization to manifolds, any manifold $M$ we consider is smooth, second countable and Hausdorff (to appeal to Whitney's embedding theory, \textit{e.g.}, see \cite[Cor. 6.16]{ref:Lee2}). Then, a set-valued vector field $F:M \rightrightarrows TM$ (\textit{e.g.}, think of $\mathcal{F}[X]$) is said to satisfy the \textit{basic conditions} when it is a locally bounded map, outer-semicontinuous and $F(x)$ is non-empty, compact and convex for all $x\in M$ \cite[Def. 5]{ref:mayhew2011topological}. These basic conditions suffice for a smooth converse Lyapunov theory \cite[Cor. 13]{ref:mayhew2011topological} via \cite{ref:CaiTeelGoebelPII2008}. In particular, we consider the following class of vector fields.

\begin{assumption}[Vector field regularity on $M$, conditioned on a point $x'\in M$]
\label{ass:reg:X:M}
Given a point $x'\in M$, our vector fields are locally essentially bounded on $M$, possibly set-valued at $x'$ and locally Lipschitz on $M\setminus(\partial M \cup\{x'\})$. 
\end{assumption}

Suppose that $X$ complies with Assumption~\ref{ass:reg:X:M}, see that after smoothly embedding $M$ into some $\mathbb{R}^k$, the vector field in new coordinates still satisfies the conditions of Assumption~\ref{ass:reg:X:M}. Then, it is known that under Assumption~\ref{ass:reg:X:M}, the differential inclusion that corresponds to the Filippov operator (\textit{i.e.}, applied after embedding $M$), satisfies the basic conditions from above. For simplicity of exposition, however, we will now directly work with embedded submanifolds $M\subseteq \mathbb{R}^k$ (generalizations are of course immediate).

Also, see that Assumption~\ref{ass:reg:X:M} only allows for mildly discontinuous vector fields $X$ akin to~\eqref{equ:normalized:flow}, \textit{i.e.}, the local boundedness is only relevant for a neighbourhood of $x'\in M$.  

Let $X$ be a vector field on an embedded submanifold $M\subseteq \mathbb{R}^k$ that complies with Assumption~\ref{ass:reg:X:M}, then, from now on, the corresponding (Filippov) \textit{solutions} $(t,p)\to \varphi^t(p)$ to the differential equation $\dot{x}=X(x)$, are understood to be absolutely continuous in $t$, on compact intervals $\mathcal{I}$, and such that they satisfy
\begin{equation*}
    \left.\frac{\mathrm{d}}{\mathrm{d}s}\varphi^s(p)\right|_{s=t}\in\mathcal{F}[X](\varphi^t(p)) \text{ for} \text{ a.e. } t\in \mathcal{I}.
\end{equation*}

We study (strong) (global) asymptotic stability of equilibria under these solutions and point to \cite[Def. 2.1]{ref:clarke1998asymptotic} and \cite[Sec. 2]{ref:mayhew2011topological} for further details. Note in particular, that under Assumption~\ref{ass:reg:X:M}, we will have uniqueness of solutions if we set $x'$ to be the attractor. This, because of stability. 

Our main result consists of two cases, with case \textit{(I)} being reminiscent of~\eqref{equ:normalized:flow} on all of $\mathbb{R}^n$ whereas case \textit{(II)} captures, for instance, \eqref{equ:normalized:flow} defined on a \textit{bounded} domain (\textit{i.e.}, this is why we consider the possibility of $\partial M \neq \varnothing$ in Assumption~\ref{ass:reg:X:M}).   

\begin{figure}
    \centering
    \includegraphics[scale=0.9]{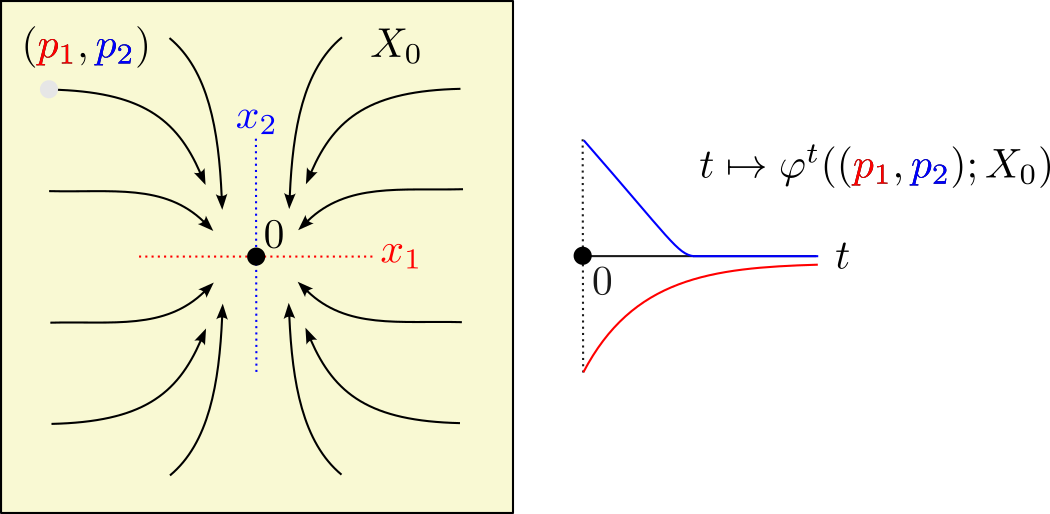}
    \caption{Some orbits corresponding to the vector field~\eqref{equ:X0}, plus a typical element-wise solution starting from a point $(p_1,p_2)\in \mathbb{R}^2$. Observe the difference in convergence (finite-time vs. asymptotic).}
    \label{fig:X0}
\end{figure}

Before stating the result, let us clarify the inherent difficulties with semiflows. Let $B(x_*)$ be the basin of attraction of $x_*\in M$ under some flow $\varphi$. In \cite{ref:kvalheim2025global}, the authors construct their linearizing homeomorphism by exploiting that $B(x_*)\setminus \{x_*\}\simeq_t \mathbb{R}\times \mathbb{S}^{n-1}$, that is, they exploit that $\mathbb{R}\times B(x_*)\subseteq\mathrm{dom}(\varphi)$. In the case of semiflows, one might define the map $T^+_*:B(x_*)\setminus\{x_*\}\to \mathbb{R}_{\geq 0}$ via $T^+_*(x):=\inf\{t\geq 0:\varphi^t(x)=x_*\}$ and hope to use $(-\infty,T^+_*(x))$ instead of $(-\infty,+\infty)$. A problem is, even under Assumption~\ref{ass:reg:X:M}, $T^+_*$ cannot be guaranteed to be continuous. For instance, consider the vector field $X_0$ on the plane, defined on $\mathbb{R}^2\setminus \{0\}$ through
\begin{equation}
\label{equ:X0}
\begin{pmatrix}
\dot{x}_1\\
\dot{x}_2
\end{pmatrix} = X_0(x):=
\begin{pmatrix}
-x_1\\
-\displaystyle\frac{x_2}{|x_2|+x_1^2}
\end{pmatrix}
\end{equation}
with $X_0(0):=0$, see Figure~\ref{fig:X0}. In this case, $T^+_*$ is discontinuous at $\{(0,x_2):x_2\neq 0\}$ (jumping from $|x_2|$ to $+\infty$, one may select $(x_1,x_2)\mapsto V(x_1,x_2)=x_1^2+x_2^2$ as a Lyapunov function to assert stability). Hence, in the following proof we will look at the smallest $t\geq 0$ such that $\varphi^t(x)$ enters a \textit{neighbourhood} of $x_*$ instead.  

In the following, \textit{closed} refers to closed in the subspace topology, not ``compact and without boundary'' as is customary in differential topology. We have visualized the theorem in Figure~\ref{fig:HGthm}. 

\begin{theorem}[A global Hartman-Grobman theorem, without hyperbolicity, for asymptotically stable semiflows]
\label{thm:discont}
For $M\subseteq \mathbb{R}^k$ a closed, smooth, $n$-dimensional embedded submanifold, let $x_*\in M$ be asymptotically stable, under a semiflow $\varphi$ (\textit{i.e.}, a Filippov solution) generated by a vector field that satisfies Assumption~\ref{ass:reg:X:M} for $x':=x_*$  and let $B(x_*)\subseteq M$ be the corresponding basin of attraction. 
\begin{enumerate}[(I)]
\item Suppose that $(-\infty,0]\times B(x_*)\setminus\{x_*\}\subseteq \mathrm{dom}(\varphi)$, then, for any $r>0$ there is a homeomorphism $h_r:B(x_*)\to \mathbb{R}^n$ and a $\gamma_r\in \mathcal{K}_{\infty}$ such that for all $y\in B(0,r)^c$ 
\begin{equation}
\label{equ:homeo:thm}
    \left.h_r\circ \varphi^t|_{B(x_*)} \circ h_r^{-1}\right|_{B(0,r)^c}(y) = \left. e^{-t \cdot I_n}\right|_{B(0,r)^c}(y)
\end{equation}
holds for all $t$ such that $0\leq t\leq \gamma_r(\|y\|_2-r)$.
\item Suppose that $(-\infty,0]\times B(x_*)\setminus\{x_*\}\not\subseteq \mathrm{dom}(\varphi)$, then, for any $r>0$ there is a homeomorphism $h_r:B(x_*)\to \mathbb{R}^n$, a $\gamma_r\in \mathcal{K}_{\infty}$ and a $R>r$ such that for all $y\in \mathrm{cl}(B(0,R))\setminus B(0,r)$, \eqref{equ:homeo:thm} holds for all $t$ such that $0\leq t\leq \gamma_r(\|y\|_2-r)$.  
\end{enumerate}
\end{theorem}

\begin{figure}
    \centering
    \includegraphics[scale=0.9]{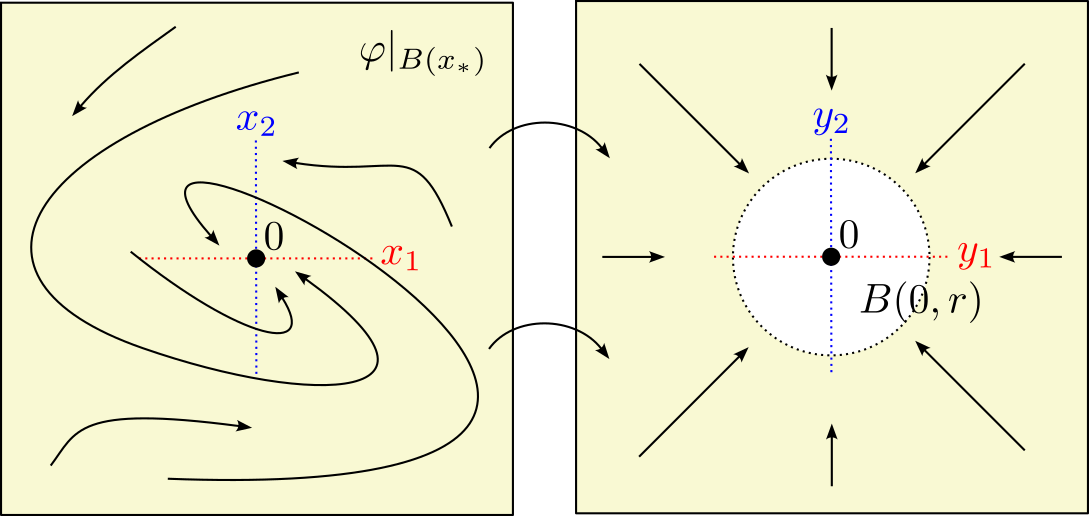}
    \caption{Visualization of Theorem~\ref{thm:discont}: any asymptotically stable semiflow (\textit{i.e.}, a semiflow such that some fixed point $x_*$ is asymptotically stable with basin of attraction $B(x_*)$) can be transformed through a topological conjugacy, ``practically'', to the canonical asymptotically stable flow corresponding to $\dot{x}=-x$, that is, the conjugacy holds outside any selected neighbourhood of $x_*$ (\textit{e.g.}, outside $B(0,r)$).}
    \label{fig:HGthm}
\end{figure}

\begin{proof}
    The proof is an extension of that of \cite[Thm. 2]{ref:kvalheim2025global}. We follow their arguments and notation to a large extent.
   
    We start with \textit{(I)}. By asymptotic stability of $x_*$ and Assumption~\ref{ass:reg:X:M}, there is a $C^{\infty}$ Lyapunov function $V:B(x_*)\to \mathbb{R}_{\geq 0}$ corresponding to the pair $(x_*,\varphi)$ \cite{ref:clarke1998asymptotic,ref:CaiTeelGoebelPII2008,ref:mayhew2011topological}.
    Now consider $V^{-1}(\varepsilon)$ for some $\varepsilon>0$. As $V^{-1}(\varepsilon)\simeq_h \mathbb{S}^{n-1}$ \cite{ref:wilson1967structure}, then, by the resolution of the \textit{topological} Poincar\'e conjecture, there is always a homeomorphism $P:V^{-1}(\varepsilon)\to \mathbb{S}^{n-1}$. The role of mapping to $\mathbb{S}^{n-1}$ is really to ``straighten'' the dynamics, the symmetry of $\mathbb{S}^{n-1}$ is less important.

    Set $L_{\varepsilon}:=V^{-1}(\varepsilon)$ and $U_{\varepsilon}:=V^{-1}([0,\varepsilon))$ to define the map $T^+_{\varepsilon}:B(x_*)\setminus U_{\varepsilon}\to \mathbb{R}_{\geq 0}$ through $T^+_{\varepsilon}(x):=\inf\{t\geq 0:\varphi^t(x)\in L_{\varepsilon} \}$. Following similar arguments as in \cite[Thm. 5]{ref:moulay2010topological}, it follows that $T^+_{\varepsilon}$ is continuous on $B(x_*)\setminus U_{\varepsilon}$. As discussed above, we cannot simply set $\varepsilon:=0$. Now, define 
    \begin{equation*}
        W:=\bigcup_{x\in L_{\varepsilon}}\left(-\infty,0\right]\times \{x\} \subseteq \mathbb{R}\times L_{\varepsilon},    
        \end{equation*}
then, thanks to continuity of $T^+_{\varepsilon}$ and the standing assumptions (\textit{i.e.}, away from $x_*$ we can move backwards indefinitely, but forwards we might converge in finite-time) it follows that $(t,x)\mapsto \varphi^t(x)$ defines a homeomorphism from $W$ to $B\setminus U_{\varepsilon}$, with inverse $g=(\tau,\rho):B\setminus U_{\varepsilon} \to W$, \textit{i.e.}, $\rho(x)\in L_{\varepsilon}$ and $\varphi^{\tau(x)}(\rho(x))=x$.

However, it will be more convenient to define the inverse $g$ through $\varphi^{-\tau(x)}(\rho(x))=x$ so that $\tau(x)\to T^+_{\varepsilon}(\rho(x))=0^+$ for $B\setminus U_{\varepsilon}\ni x\to L_{\varepsilon}$. 

The homeomorphism $h$ from \cite{ref:kvalheim2025global}, as defined through $x\mapsto h(x):=e^{\tau(x)}P(\rho(x))$, was designed for flows and relies on $\tau(x)\to -\infty$ for $x\to x_*$. In particular, their domain of $
\tau$ is $B(x_*)\setminus \{x_*\}$ and not $B(x_*)\setminus U_{\varepsilon}$. Thus, we need to rescale. In particular, consider the homeomorphism $x\mapsto h'(x):=e^{\tau'(x)}P(\rho'(x))$ from $B(x_*)$ to $\mathbb{R}^n$, with ${\tau'}$ defined on $B(x_*)\setminus \{x_*\}$ through 
\begin{equation}
\label{equ:tau:prime}
    \tau'(x) := \begin{cases}
\tau(x) \quad &x\in B(x_*)\setminus U_{\varepsilon}\\
\displaystyle \ln\left(\frac{V(x)}{\varepsilon}\right)\quad &\text{otherwise}.
    \end{cases}
\end{equation}
Now, define the map $T^-_{\varepsilon}:(L_{\varepsilon}\cup U_{\varepsilon})\setminus\{x_*\}\to \mathbb{R}_{\geq 0}$ through $T^-_{\varepsilon}(x)=\inf\{t\geq 0:\varphi^{-t}(x)\in L_{\varepsilon}\}$, and eventually $\rho'$ on $B(x_*)\setminus\{x_*\}$ via
\begin{equation*}
\label{equ:rho:prime}
    \rho'(x) := \begin{cases}
\rho(x) \quad &x\in B(x_*)\setminus U_{\varepsilon}\\
\varphi^{-T^-_{\varepsilon}(x)}(x)\quad &\text{otherwise}.
    \end{cases}
\end{equation*}
At last, set $h'(x_*):=0$. Regarding continuity, observe that for $x\in L_{\varepsilon}$, we have $\tau(x)=0$ and for $U_{\varepsilon}\ni x\to L_{\varepsilon}$ we have $\tau(x)\to 0^-$. Moreover, for $x\to x_*$ we have that $\tau'(x)\to -\infty$ and thus $e^{\tau'(x)}\to 0$ such that thanks to compactness of $L_{\varepsilon}$ we have $h(x)\to 0$. Continuity of $T_{\varepsilon}^-$ follows again from \cite[Thm. 5]{ref:moulay2010topological} and continuity of $h'^{-1}$ follows, for instance, from an open mapping argument (\textit{e.g.}, see \cite[Thm. A.38]{ref:Lee2}).  

Following \cite{ref:kvalheim2025global}, see that for any $x\in B(x_*)\setminus U_{\varepsilon}$ we have $\rho'(\varphi^t(x))=\rho'(x)$ and for sufficiently small $t\geq 0$ we have $\tau(\varphi^t(x))=\tau(x)-t$ (\textit{e.g.}, see that $\varphi^t(x)=\varphi^t(\varphi^{-\tau(x)}(\rho(x)))=\varphi^{-\tau(\varphi^t(x))}(\rho(x))$.

Specifically, see that $\tau'(\varphi^t(x))=\tau(x)-t$ for $x\in B(x_*) \setminus U_{\varepsilon}$ and $t\leq T^+_{\varepsilon}(x)$ with $T^+_{\varepsilon}(x)>0$ for $x\in B(x_*)\setminus (L_{\varepsilon}\cup U_{\varepsilon})$.

Now, let $y:=h'(x)$, then $x\in L_{\varepsilon}\implies y=h'(x)\in \mathbb{S}^{n-1}$. Hence, \eqref{equ:tau:prime} implies that outside of $B(0,1)$, the semiflow of the conjugate system, that is, $(t,y)\mapsto h'\circ \varphi^t|_{B(x_*)}\circ h'^{-1}(y)$, is equivalent to $(t,y)\mapsto e^{-t}y$, for $t\leq T^+_{\varepsilon}(x)=T^+_{\varepsilon}(h'^{-1}(y))$ and $y\in B(0,1)^c$.
To recover $h_r$ of the theorem, simply scale $h'$ by $r$ to scale the radius of the ball (recall that $P:L_{\varepsilon}\to \mathbb{S}^{n-1}$). 

At last, define $\gamma_r:\mathbb{R}_{\geq 0}\to \mathbb{R}_{\geq 0}$ through 
\begin{equation*}
s\mapsto\gamma_r(s) := \inf_{y\in \mathbb{R}^n} \left\{T^+_{\varepsilon} \circ \left.h_r^{-1}\right|_{B(0,r)^c}(y):\|y\|_2 = s+r\right\}.
\end{equation*}
It follows from, for instance, Berge's maximum theorem \cite[p. 115]{ref:berge1963topological} that $\gamma_r$ is continuous and from the fact that on $B(x_*)\setminus \{x_*\}$, $\varphi^t(x)$ extends indefinitely backwards, that $\gamma_r\to +\infty$ for $\|y\|_2\to +\infty$. Hence, as by properties of $\varphi, V$ and $T^+_{\varepsilon}$, $\gamma_r$ is strictly monotone, we have that $\gamma_r\in \mathcal{K}_{\infty}$.

To continue with \textit{(II)}, the arguments are almost identical, yet, now we need to carefully deform $B(x_*)\setminus \{x_*\}$ around all its ``\textit{boundaries}'', not only around $x_*$. Although we keep the same notation as for \textit{(I)}, one should understand the Lyapunov function $V$ in the context of \textit{(II)}.

To start, pick some $C>\varepsilon$ and set $L_C:=V^{-1}(C)$ and $U_C:=V^{-1}([0,C))$ to define the map $T^-_{C}: (L_C\cup U_C)\setminus U_{\varepsilon} \to \mathbb{R}_{\geq 0}$ through $T^-_{C}(x) = \inf_t\{t\geq 0:\varphi^{-t}(x)\in L_C\}$. Now, if we define  
    \begin{equation*}
        W_{C}:=\bigcup_{x\in L_{\varepsilon}}[-T^-_{C}(x),0]\times \{x\} \subseteq \mathbb{R}\times L_{\varepsilon},    
        \end{equation*}
then, $(t,x)\mapsto \varphi^t(x)$ yields a homeomorphism from $W_{C}$ to $(L_C\cup U_C)\setminus U_{\varepsilon}\subset B(x_*)\setminus \{x_*\}$.  

To proceed, we will effectively ``\textit{linearize}'' on this compact set $K_{\varepsilon,C}:= (L_C\cup U_C)\setminus U_{\varepsilon}=V^{-1}([\varepsilon,C])$, accommodated by appropriate deformations on the remaining space. 

With this construction in mind, consider now the homeomorphism $x\mapsto h_C(x):=e^{\tau_C(x)}P(\rho_C(x))$ from $B(x_*)$ to $\mathbb{R}^n$, with again $h_C(x_*):=0$, $\rho_C:=\rho'$ and with ${\tau_C}$ defined on $B(x_*)\setminus \{x_*\}$ through 
\begin{equation*}
    \tau_C(x) := \begin{cases}
    \tau(x)+\left( V(x)-V(C)\right)\quad & x\in B(x_*)\setminus K_{\varepsilon,C}\\
\tau(x) \quad &x\in K_{\varepsilon,C}\\
\displaystyle\ln \left(\frac{V(x)}{\varepsilon} \right) \quad &\text{otherwise}.
    \end{cases}
\end{equation*}
Note that by coercivity of $V$, we have $\tau_C(x)\to +\infty$ for $x\to \partial B(x_*)$, with $x\in B(x_*)\setminus \{x_*\}$. 

Again, we can multiply $h_C$ by $r$ to rescale. In this case, $R$ is given by $\inf_{x\in L_C}r\cdot e^{\tau_C(x)}$, which is attained due to compactness of $L_C$ and continuity of $\tau_C$. As $\tau_C(x)=0$ for $x\in L_{\varepsilon}$ and $C>\varepsilon$, we have that $R>r$. Note, this exponential term defining $R$ is exactly what one should expect, given that we do not reparametrize time (\textit{e.g.}, we should have $r=e^{-T}R$ for some appropriate $T>0$). 

At last, we can also employ $T^+_{\varepsilon}$ and $\gamma_r$ again, subject to constraining their domain appropriately.   
\end{proof}

Of course, if $\varphi$ happens to be a flow (\textit{i.e.}, $t\in \mathbb{R}$), Theorem~\ref{thm:discont} is also true, but slightly conservative as there is no need to construct $B(0,r)$ and the like~\textit{cf}. \cite[Thm. 2]{ref:kvalheim2025global}. 

We also emphasize that we merely show that there is a \textit{homeomorphism} $h_r$, one should not expect to be able to differentiate through $h_r$, \textit{e.g.}, see Section~\ref{sec:ex} and Figure~\ref{fig:HGflownormtf}.  

\begin{remark}[Asymptotically stable semiflows are practically indistinguishable]
Theorem~\ref{thm:discont} shows that, up to a change of coordinates, every asymptotically stable semiflow is ``practically'' topologically conjugate to the canonical semiflow $(y,t)\mapsto e^{-t}y$. Here, ``practically'' means that for any (\textit{e.g.}, arbitrarily small) neighbourhood $U\subseteq B(x_*)$ of $x_*$, we can find a homeomorphism that gives rise to the conjugacy on $B(x_*)\setminus U$ for sufficiently small $t$. This means that, up to a change of coordinates, all these asymptotically stable semiflows are practically indistinguishable. This seems especially interesting when the true system is finite-time stable and we happen to work with coordinates that give the impression (\textit{i.e.}, on $B(0,r)^c$) that the stability is merely asymptotic, see Section~\ref{sec:ex:A} below. Clearly, the converse is also interesting, as illustrated in Section~\ref{sec:ex:B}. We believe these observations have ramifications for data-driven learning of stable dynamical systems.
\end{remark}

\begin{figure}
    \centering
    \includegraphics[scale=0.9]{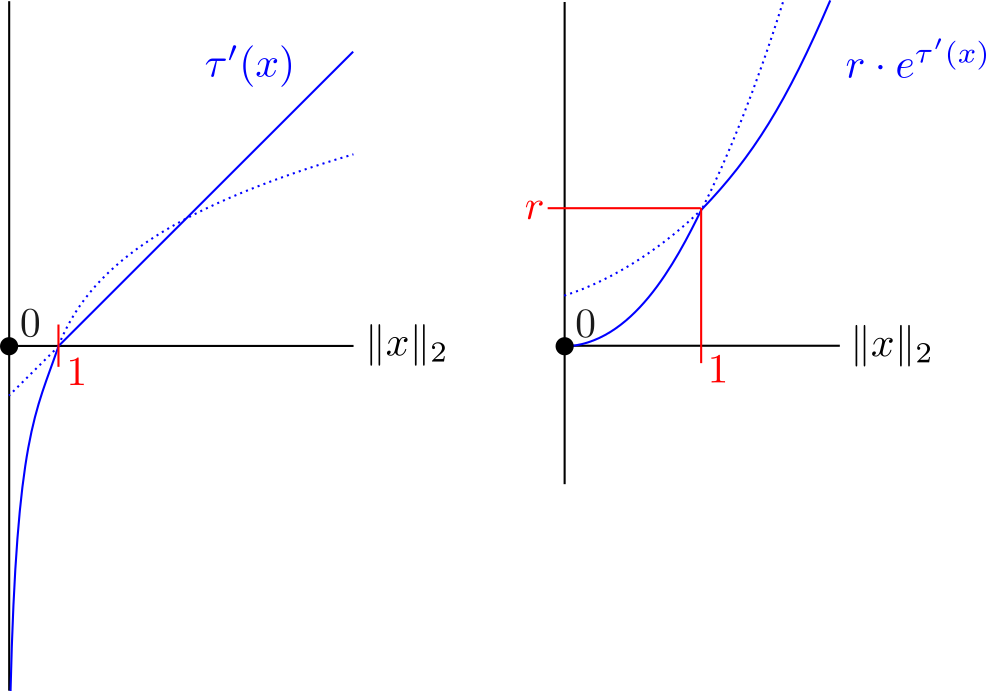}
    \caption{The maps $x\mapsto \tau'(x)$ and $x\mapsto r\cdot e^{\tau'(x)}$ from Section~\ref{sec:ex:A}. Recall that we have set $\varepsilon$ such that $V^{-1}(\varepsilon)=L_{\varepsilon}=\mathbb{S}^{n-1}$ such that $\tau'(x)=\ln(\|x\|_2^2)$ for $\|x\|_2<1$ and $\tau'(x)=\|x\|_2-1$ for $\|x\|_2\geq 1$.}
    \label{fig:HGtau}
\end{figure}

\section{Examples}
\label{sec:ex}
To exemplify Theorem~\ref{thm:discont} and utilize the semi-constructive nature of its proof, we discuss case \textit{(I)} (as case \textit{(II)} follows as a corollary). For the notation, we refer to Theorem~\ref{thm:discont} and its proof. 

\subsection{}
\label{sec:ex:A}
Specifically, regarding~\eqref{equ:normalized:flow} on $\mathbb{R}^n$, note that $x\mapsto V(x):=\frac{1}{2}\|x\|_2^2$ checks out as a Lyapunov function\footnote{Indeed, this Lyapunov function even certifies finite-time stability as $\dot{V}\leq -\frac{1}{2}V^{\alpha}$ for $\alpha=\tfrac{1}{2}$, \textit{e.g.}, see \cite{ref:MP2005finite}.} so that $B(x_*)=\mathbb{R}^n$ and $x_*=0$. Indeed $V^{-1}(\varepsilon)=\{x\in \mathbb{R}^n:\|x\|_2=\sqrt{2\varepsilon}\}$, so we simply select $\varepsilon:=\frac{1}{2}$ such that $P:V^{-1}(\varepsilon)\to \mathbb{S}^{n-1}$ becomes the identity map on $\mathbb{S}^{n-1}$. 

As, $T^-_{\varepsilon}(x)=1-\|x\|_2$, we get that $x\mapsto \rho'(x)=x/\|x\|_2$ for all $x\in B(x_*)\setminus\{x_*\}$. Then since $\varphi_1^{-\tau(x)}(\rho(x))=x$ we get that $\tau'(x)=\|x\|_2-1$ on $B(x_*)\setminus U_{\varepsilon}$ and $\tau'(x)=\ln(\|x\|_2^2/ 2\varepsilon)$ on $U_{\varepsilon}\setminus \{x_*\}$. See Figure~\ref{fig:HGtau} for a visualization of the map $\tau'$.

Now, the homeomorphism $h_r:\mathbb{R}^n\to \mathbb{R}^n$ from Theorem~\ref{thm:discont} is explicitly given as
\begin{equation*}
    h_r(x) := \begin{cases}
        \displaystyle r\cdot e^{\tau'(x)}\frac{x}{\|x\|_2}\quad &x\in \mathbb{R}^n\setminus \{0\}\\
        0 \quad & \text{otherwise}.
    \end{cases} 
\end{equation*}
Indeed, for $x\to x_*$ we have that $e^{\tau'(x)}\to 0$. We can also explicitly state the inverse $h_r^{-1}:\mathbb{R}^n\to \mathbb{R}^n$ as

\begin{equation*}
    h^{-1}_r(y) := \begin{cases}
        \displaystyle \alpha_r(y)\frac{y}{\|y\|_2}\quad &y\in \mathbb{R}^n\setminus \{0\}\\
        0 \quad & \text{otherwise}
    \end{cases} 
\end{equation*}
with $\alpha_r$ defined on $\mathbb{R}^n\setminus \{0\}$ through
\begin{equation*}
    \alpha_r(y) :=\begin{cases}
        \displaystyle\ln\left(\frac{\|y\|_2}{r}\right)+1 \quad & y\in B(0,r)^c\\
        \displaystyle \left(\frac{\|y\|_2}{r}\right)^{1/2}\quad &\text{otherwise}.
    \end{cases}
\end{equation*}

Now we have the ingredients to explicitly compute $h_r\circ \varphi_1^t\circ h_r^{-1}$. We do this for (i) $y\in B(0,r)^c$ and (ii) for $y\in B(0,r)$.

(i)~First, suppose that we have $y\in B(0,r)^c$ such that $\|y\|_2=\delta r$ for some $\delta\geq 1$. It follows that $\alpha_r(y) = 1+\ln(\delta)$ and thus $x:=h_r^{-1}(y)=(1+\ln(\delta))y/\|y\|_2$. Recall that $\|\varphi^t(x')\|_2 = \|x'\|_2-t$, for any $x'\in \mathbb{R}^n\setminus \{0\}$ and $t\leq \|x'\|_2$ and thus, we have that $\varphi^t(x)\in B(0,1)^c$ for $t\leq \ln(\delta)$. Then, looking at the definition of $\tau'$, we can put it all together, that is, we get
\begin{equation*}
    h_r\circ \varphi^t\circ h_r^{-1}(y) = r\delta e^{-t} \frac{y}{\|y\|_2} = e^{-t}y
\end{equation*}
for $t\leq \ln(\delta)=\ln(\|y\|_2/r)$. Therefore, under the homeomorphism $h_r$, once we start in $B(0,r)^c$, we flow towards $B(0,r)$ along the canonical dynamics $\dot{x}=-x$. 

We add that a simple computation shows that $\gamma_r\in \mathcal{K}_{\infty}$, as constructed in the proof of Theorem~\ref{thm:discont}, becomes
\begin{equation*}
s\mapsto \gamma_r(s):=\ln \left(\frac{s+r}{r} \right).
\end{equation*}
Indeed, then $\gamma_r(\|y\|_2-r)=\ln(\delta)$. 

(ii)~Now suppose we start within $B(0,r)$, say $\|y\|_2=\theta r$ with $\theta\in (0,1)$ ($\theta=0$ is not very interesting). To comply with topological conjugacies, we must recover finite-time stability around the stable equilibrium. See that $a_r(y;\theta)=\sqrt{\theta}$ and thus $x:=h_r^{-1}(y)=\sqrt{\theta} y/\|y\|_2$ with $\|x\|_2<1$. This means that we look at the second case of $\tau'$ and 
\begin{equation}
\label{equ:ii:conjugacy}
\begin{aligned}
    h_r\circ \varphi^t\circ h_r^{-1}(y) &= r \left(\sqrt{\theta}-t\right)^2 \frac{y}{\|y\|_2}\\
         &= r \left(\sqrt{\|y\|_2/r}-t\right)^2 \frac{y}{\|y\|_2}
\end{aligned}
\end{equation}
for $t\leq \sqrt{\theta}=\sqrt{\|y\|_2/r}$. Thus, finite-time stability is preserved, as it should. Comparing \eqref{equ:ii:conjugacy} to \eqref{equ:flow:sol}, see that apparent scaling is due to our \textit{choice} of Lyapunov function. Note that the dynamics on $B(0,r)$ are now reminiscent of $\dot{s}=-\sqrt{s}$, for $s\geq 0$, as the solution to this positive, scalar ODE becomes $(s,t)\mapsto\varphi_{\sqrt{}}^t(s):=( (s-t)/2)^2$ for $t\leq s$.

With the above in mind, let $\widetilde{\varphi}_1$ denote the transformed semiflow, \textit{i.e.}, $\widetilde{\varphi}_1^t := h_r\circ \varphi_1^t\circ h_r^{-1}$, then we visualize the transformed dynamics in Figure~\ref{fig:HGflownormtf}. Observe the exponential decay until $B(0,r)$. 

\begin{figure}
    \centering
    \includegraphics[scale=0.9]{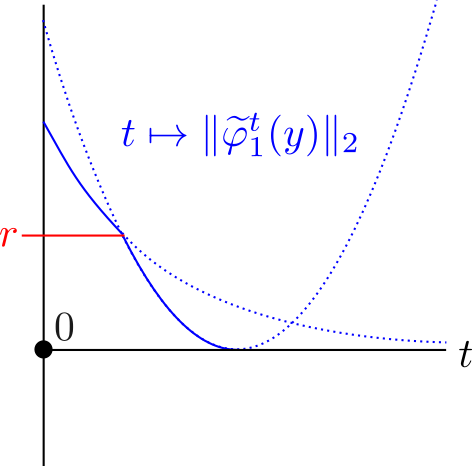}
    \caption{Resulting semiflow from Section~\ref{sec:ex:A}, \textit{e.g.}, we visualize $t\mapsto \|\widetilde{\varphi}_1^t(y)\|_2$ for $r=1$ and starting from $\|y\|_2=2$, that is, although the true system is finite-time stable, we first have exponential (asymptotic) decay for $t\leq \ln(2)$, then we follow \eqref{equ:ii:conjugacy} (from $B(0,r)$) and observe a decay of the form $t\mapsto (1-(t-\ln(2))^2$ until we hit $0$.}
    \label{fig:HGflownormtf}
\end{figure}

\subsection{}
\label{sec:ex:B}
Suppose we work with $\dot{x}=-x$ and see that~\eqref{equ:normalized:flow} can be understood as the ``canonical'' finite-time stable system. The previous section detailed how to go from~\eqref{equ:normalized:flow} to the flow under $\dot{x}=-x$, at least, on $B(0,r)^c$. Evidently, this can be reversed. 
We will provide a simple example. 

In Section~\ref{sec:ex:A} we constructed the map $P:L_{\varepsilon}=V^{-1}(\varepsilon)\to \mathbb{S}^{n-1}$ to be the identity, which in turn led to $h_r^{-1}(r\cdot\mathbb{S}^{n-1})=\mathbb{S}^{n-1}$, for $h_r$ as in Section~\ref{sec:ex:A}. Indeed, one may further rescale these spheres, however, we continue without doing so and readily find the following: 
\begin{equation}
\label{equ:homeo:ex:B}
    \left.h_r^{-1}\circ e^{t\cdot I_n} \circ h_r\right|_{B(0,1)^c}(x) = \left. \varphi_1^{t}\right|_{B(0,1)^c}(x)
\end{equation}
for $t\leq \ln(\|h_r(x)\|_2/r)=\tau'(x)$. We visualize $h_r^{-1}\circ e^{-t\cdot I_n}\circ h_r=:\widetilde{\varphi}^t_e$ in Figure~\ref{fig:HGrev}. 

\section{Discussion}

\begin{figure}
    \centering
    \includegraphics[scale=0.9]{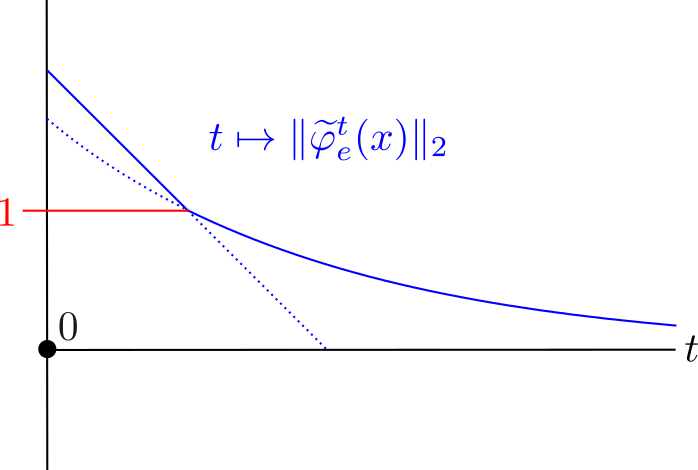}
    \caption{Resulting semiflow from Section~\ref{sec:ex:B}, \textit{e.g.}, we visualize $t\mapsto \|\widetilde{\varphi}_e^t(x)\|_2$ starting from $\|x\|_2=2$, that is, although the true system is merely asymptotically stable, we follow $t\mapsto\|x\|_2-t$ (\textit{i.e.}, typical finite-time behaviour) until we hit the unit ball. Afterwards we have an asymptotic decay of the form $t\mapsto e^{-1/2 (t-1)}\|x\|_2$.}
    \label{fig:HGrev}
\end{figure}

There are several ways to prove Theorem~\ref{thm:discont} and similar results. We did not consider a reparametrization of time, nor the simplest form the dynamics could have outside of the domain of linearization. Plus, for simplicity, we have focused on standard Euclidean balls (\textit{e.g.}, $B(0,r)$).

In spirit, Theorem~\ref{thm:discont} can also be understood as a non-local (and non-constant) version of flow-box (or straightening-out) results \textit{cf}. \cite[Thm. 2.26]{ref:nonlin}. Better yet, one may interpret the above as a notion of a \textit{practical linearization}, akin to practical stability \cite[\S 25]{ref:lefschetzlasalle1961} that is.  

Results like these not only reinforce the Koopman operator framework, but also the ``\textit{hybridization}'' of topological index theory, \textit{e.g.}, see \cite{ref:gottlieb1995index,ref:Kvalheim2021PHhybrid}. Moreover, these results are closely related to recent studies \cite{ref:jongeneel2024hGAS,ref:kvalheim2025differential} into the topology of stable systems as decompositions like \eqref{equ:homeo:thm} allow for studying spaces of stable systems through homeomorphism groups, \textit{e.g.}, see \cite[Prop. 3.2]{ref:jongeneel2024hGAS}. To reiterate, although we allude to links with Koopman operator theory, we emphasize that the Koopman framework is more generally concerned with spectra, not continuous transformations per se \cite{ref:mezic2021koopman}. To demand that a linearization itself is a \textit{continuous map} is something we are interested in due to aforementioned links with topological dynamics systems theory. Regarding future work, we are especially interested in generalizations beyond equilibrium points, that is, to general attractors. Even when restricted to compact attractors, this extension is highly non-trivial. The ``canonical system'' (\textit{e.g.}, $\dot{x}=-x$ on $\mathbb{R}^n$) is not obvious, although for some compact attractors $A$ on a metric space $(X,d)$, one might consider an ODE of the form $\dot{x} = -\partial_x d(x,A)^2$, for $d(x,A)=\inf_{a\in A}d(x,a)$. However, despite a well-developed converse Lyapunov theory (\textit{e.g.}, see~\cite{ref:bhatia1970stability}), it is yet unclear how to generalize the approach from above. Note, we do have a well-developed understanding of homotopy equivalence \cite{ref:wilson1967structure,ref:lin2022wilson} and topological embeddings \cite{ref:kvalheim2023linearizability}. Additionally, we know that not every compact attractor itself is homotopy equivalent to its domain of attraction~\cite{ref:jongeneelECC24}, thus obstructing a direct extension, in general. For merely closed attractors, the situation is even less transparent \cite{ref:lin2022wilson,ref:jongeneel2025topological}.  

At last, we are also interested in deeper ramifications for learning dynamical systems, as alluded to before. 
\printbibliography
\end{document}